\documentclass[12pt]{amsart}

\usepackage{amsmath}
\usepackage{amsfonts,amssymb}
\usepackage{enumitem}
\usepackage{color,linegoal}
\usepackage{hyperref}
\usepackage{MnSymbol}

\newcounter{commentcounter}
\definecolor{SolutionColor}{rgb}{1,1,0.5}

\usepackage{stmaryrd}

\newcounter{commentscounter}
\usepackage{graphicx}

\theoremstyle{plain}
\newtheorem*{theorem*}{Theorem}
\newtheorem*{lemma*} {Lemma}
\newtheorem*{corollary*} {Corollary}
\newtheorem*{proposition*} {Proposition}
\newtheorem{theorem}{Theorem}[section]
\newtheorem{lemma}[theorem]{Lemma}
\newtheorem{corollary}[theorem]{Corollary}

\theoremstyle{remark}
\newtheorem*{remark}{Remark}
\newtheorem*{definition}{Definition}

\newtheorem{example}[theorem]{Example}

\theoremstyle{definition}

\def \Z {\Bbb{Z}}
\def \C {\Bbb{C}}

\def\op{\operatorname}

\def\wti{\widetilde}
\def\sl{\op{SL}}
\def\op{\operatorname}

\def\gl{\op{GL}}

\def\Z{\Bbb{Z}}
\def\C{\Bbb{C}}
\def\N{\Bbb{N}}

\def\part{\partial}

\def\a{\alpha}

\def\bp{\begin{pmatrix}}

\def\sm{\setminus}

\def\ep{\end{pmatrix}}
\def\bn{\begin{enumerate}}

\def\en{\end{enumerate}}
\def\ba{\begin{array}}
\def\ea{\end{array}}

\def\a{\alpha}

\def\ti{\tilde}

\def\fr12{\frac{1}{2}}

\def\be{\begin{equation} }
 \def\ee{\end{equation}}

\def\co{\colon\,}

\def\cmtbf#1{} \def\cmt#1{}

\def\op{\operatorname}

\def\tihol{\wti{\op{Hol}}}
\DeclareMathOperator{\SL}{SL}
\DeclareMathOperator{\GL}{GL}
\DeclareMathOperator{\ind}{ind}
\DeclareMathOperator{\res}{res}

\begin{document}

\title{On high-dimensional representations of knot groups}

\author{Stefan Friedl}
\address{Fakult\"at f\"ur Mathematik\\ Universit\"at Regensburg\\93040 Regensburg\\   Germany}
\email{sfriedl@gmail.com}

\author{Michael Heusener}
\address{
Universit\'e Clermont Auvergne, Universit\'e Blaise Pascal, Laboratoire de Math\`ematiques, 
BP 10448, F-63000 Clermont-Ferrand \\ CNRS, UMR 6620, LM, F-63178 Aubiere, France}
\email{Michael.Heusener@math.univ-bpclermont.fr} 
\date{\today}

\begin{abstract}
Given a hyperbolic knot $K$ and any $n\geq 2$ the abelian representations and the holonomy representation each give rise to an $(n-1)$-dimensional component in the $\SL(n,\C)$-character variety.
A component of the $\SL(n,\C)$-character variety of dimension $\geq n$ is called high-dimensional.

It was proved by Cooper and Long that there exist  hyperbolic 
knots with high-dimensional components in the $\SL(2,\C)$-character variety. We show that given any non-trivial knot $K$ and sufficiently large $n$ the $\SL(n,\C)$-character variety of $K$ admits high-dimensional components.
 \end{abstract}

\maketitle


\section{Introduction}

Given a knot $K\subset S^3$ we denote by $E_K=S^3\sm \nu K$ the knot exterior and we write 
$\pi_K=\pi_1(E_K)$. Furthermore, given a group $G$ and $n\in \N$ we denote by $X(G,\sl(n,\C))$ the $\sl(n,\C)$-character variety. We recall the precise definition in Section~\ref{section:definition-character-variety}.
It is straightforward to see that the abelian representations of a knot group $\pi_K$ give rise to an $(n-1)$-dimensional subvariety of $X(\pi_K,\sl(n,\C))$ consisting solely of characters of abelian representations (see \cite[Sec.~2]{HMP15}).

If $K$ is hyperbolic, then we denote by $\tihol\co \pi_K\to \sl(2,\C)$ a lift of the holonomy representation. For $n\geq 2$ we denote by 
$\zeta_n\co \sl(2,\C)\to \sl(n,\C)$ the, up to conjugation, 
 unique rational irreducible representation of $\sl(2,\C)$.
Menal-Ferrer and Porti \cite{MFP12a,MFP12b} showed that for any $n$ the 
representation $\rho_n := \zeta_n\circ\tihol$ is a smooth point of the $\SL(n,\C)$-representation 
variety $R(\pi_k,\sl(n,\C))$. Moreover, Menal-Ferrer and Porti proved that the corresponding character $\chi_{\rho_n}$ is a smooth point on the character variety $X(\pi_K,\sl(n,\C))$, it is contained in a unique component of dimension $n-1$ \cite[Theorem~0.4]{MFP12a}. Also, the deformations of reducible representations studied in \cite{HM14} and \cite{BenAH15} give rise to $(n-1)$-dimensional components in the character variety $X(\pi_K,\sl(n,\C))$.

The above discussion shows that given any knot $K$ the character variety $X(\pi_K,\sl(n,\C))$ contains an $(n-1)$-dimensional subvariety consisting of abelian representations and if $K$ is hyperbolic $X(\pi_K,\sl(n,\C))$ also contains an $(n-1)$-dimensional subvariety that contains characters of irreducible representations.

This motivates the following definition. Given a knot $K$ we say that a component of  $X(\pi_K,\sl(n,\C))$ is \emph{high-dimensional} if its dimension is  greater than $n-1$. We summarize some known facts about the existence and non-existence of high-dimensional components of character varieties of knot groups.
\begin{itemize} 
\item For $n=3$ and $K$ a non-alternating torus knot the variety $X(\pi_K,\sl(3,\C))$ has $3$-dimensional components, whereas  for alternating torus knots  $X(\pi_K,\sl(3,\C))$ has only 
$2$-dimensional components. In particular for alternating torus knots $X(\pi_K,\sl(3,\C))$ does not contain any high-dimensional components. For more details see \cite[Thm.~1.1]{MunPor16}.
\item For $n=3$ and $K=4_1$ the variety $X(\pi_K,\sl(3,\C))$ has five $2$-dimensional components.
Three of the five components contain characters of irreducible representations. There are no higher-dimensional components. (See \cite[Thm.~1.2]{HMP15}.)
\item
It was proved by Cooper and Long \cite[Sec.~8]{CL96} that for a given $n$ there exists an alternating hyperbolic knot $K_n$ in $S^3$ such that the $\SL(2,\C)$-character variety admits a component of dimension at least $n$. 
\end{itemize}
The main result of this note is to prove that the $\SL(n,\C)$-character variety of every non-trivial knot admits high-dimensional components for $n$ sufficiently large.

\begin{theorem}\label{thm:main}
Let $K\subset S^3$ be a non-trivial
knot. Then for all $N\in\N$ there exists an $n\geq N$ such that 
the character variety $X(\pi_K,\sl(n,\C))$ contains a high-dimensional component.
\end{theorem}

Given a group $G$ we now denote by $X^\mathit{irr}(G,\sl(n,\C))$ the character variety corresponding to irreducible representations. We refer to Section~\ref{section:proof-irreducible-reps} for the precise definition. The following is now a more refined version of Theorem~\ref{thm:main}.

\begin{theorem}\label{mainthm2}
Let $K\subset S^3$ be a non-trivial knot. Then given any $N\in \N$ there exists a
$p\geq N$,  such that 
$X^\mathit{irr}(\pi_K,\sl(p,\C))$ contains a high-dimensional component.
\end{theorem}

In the special case of the figure-eight knot we obtain a refined quantitative result:
\begin{corollary}\label{cor:figure-8}
Let $K\subset S^3$ be the figure-eight knot. Then for all $n\in \N$ the representation variety
$X(\pi_K,\sl(10n,\C))$ has a component $C$ of dimension at least $4n^2-1$.
Moreover, $C$ contains characters of irreducible representations.
\end{corollary}

\begin{remark}
For a free group $F_r$ we have: $\dim X(F_r,\sl(n,\C))/(n^2-1) = (r-1)$, and hence
$ \limsup\limits_{n\to\infty} \dim X(G,\sl(n,\C))/(n^2-1) \leq (r-1)$ if $G$ is generated by $r$ elements.
It follows from Corollary~\ref{cor:figure-8} that for the figure-eight knot $K=4_1$
\[
1/25 \leq \limsup\limits_{n\to\infty} \big(\dim X(\pi_K,\sl(n,\C))/(n^2-1)\big)\leq 1
\] 
holds.
\end{remark}

\section{Representation and character varieties}
\label{section:definition-character-variety}

Before we provide the proof of Theorem~\ref{thm:main} we recall some definitions and facts.
The general reference for representation and character varieties is 
 Lubotzky's and Magid's book \cite{LM85}. 
 
 Given two representations $\rho_1\co G \to\mathrm{GL}(n_1,\C)$ and
$\rho_2\co G\to\mathrm{GL}(n_2,\C)$ we define the \emph{direct sum}
$\rho_1\oplus\rho_2\co G\to\mathrm{GL}(n_1+n_2,\C)$ 
by
\[
\big(\rho_1\oplus\rho_2\big)(\gamma) = 
\left(\begin{array}{c|c} \rho_1(\gamma) & 0 \\ \hline 0 &\rho_2(\gamma)\end{array}\right)\,.
\]
\begin{definition}\label{def:irreducibleRep}
We call a representation $\rho\co G \to\mathrm{GL}(n,\C)$ \emph{reducible} if there exists a nontrivial subspace $V\subset\C^n$, $0\neq V\neq \C^n$, such that $V$ is 
$\rho(G)$-stable. The representation 
$\rho$ is called \emph{irreducible} or \emph{simple}
 if it is not reducible.
A \emph{semisimple} representation is a direct sum of simple representations.
\end{definition}

 Let $G=\langle g_1,\ldots,g_r\rangle$ be a finitely generated group. 
A $\SL(n,\C)$-representation is a homomorphism 
$\rho\colon G \to\SL(n,\C)$.
The 
$\SL(n,\C)$-representation variety is
\[
R( G,\sl(n,\C))\,\, =\,\, \mathrm{Hom} ( G,\SL(n,\C))
\,\,\subset \,\,
\SL(n,\C)^r\,\,\subset\,\, M_n(\C)^r\,\,\cong\,\,\C^{n^2r}\,.
\]
The representation variety 
$R( G,\sl(n,\C))$ is an affine algebraic set. It is contained in
$\SL(n,\C)^r $ via the inclusion 
$\rho\mapsto\big(\rho(g_1),\ldots,\rho(g_r)\big)$, and
it is
the set of solutions of a system of polynomial equations in the matrix coefficients.

The group $\SL(n,\C)$ acts by conjugation on $R(G,\sl(n,\C))$.
More precisely, for $A\in\SL(n,\C)$ and $\rho\in R(G,\sl(n,\C))$ we define
$(A.\rho) (g) = A \rho(g) A^{-1}$ for all $g\in G$. 
In what follows we will write $\rho\sim\rho'$ if there exists 
an $A\in\SL(n,\C)$ such that $\rho'=A.\rho$, and we will call $\rho$ and $\rho'$ \emph{equivalent}.
For $\rho\in R( G,\sl(n,\C))$ we define its \emph{character}
$\chi_\rho\co G\to\C$ by $\chi_\rho(\gamma)= \mathrm{tr}(\rho(\gamma))$.
We have $\rho\sim\rho' \Rightarrow\chi_\rho = \chi_{\rho'}$.
Moreover, if $\rho$ and $\rho'$ are semisimple, then $\rho\sim\rho' $ if and only if $\chi_\rho =\chi_{\rho'}$. 
(See Theorems~1.27 and  1.28 in Lubotzky's and Magid's book \cite{LM85}.)

The \emph{algebraic quotient} or \emph{GIT quotient} for the action of $\SL(n,\C)$ on
$R(G,\sl(n,\C))$ is called the \emph{character variety}. This quotient  will be denoted by
$X(G,\sl(n,\C)) = R(G,\sl(n,\C))\sslash\SL(n,\C)$.
The character variety is not necessarily an irreducible affine algebraic set.
Work of C.~Procesi \cite{Pro76} implies that there exists a finite number of group elements
$\{\gamma_i\mid 1\leq i \leq M\}\subset G$ such that the image of 
$t\co R(G,\sl(n,\C))\to\C^M$ given by 
\[t(\rho) = \big(\chi_\rho(\gamma_1),\ldots,\chi_\rho(\gamma_M)\big)\] 
can be identified with the affine algebraic set $X(G,\sl(n,\C))\cong t(R(G,\sl(n,\C)))$,
see also \cite[p.~27]{LM85}. This justifies the name \emph{character variety}.
For an introduction to algebraic invariant theory see Dolgachev's book \cite{Dol04}.
For a brief introduction to $\SL(n,\C)$-representation and character varieties of groups see \cite{Heu16}.

\begin{example} \label{ex:free}
For a free group $F_r$ of rank $r$ we have
$R(F_r,\sl(n,\C))\cong \SL(n,\C)^r$ is an irreducible algebraic variety of dimension
$r(n^2-1)$, and the dimension of the character variety $X(F_k,\SL(n,\C))$ is $(r-1)(n^2-1)$.
\end{example}

The first homology group of the knot exterior is isomorphic to $\Z$.
A canonical surjection
$\varphi\co \pi_K\to\Z$ is given by $\varphi(\gamma)=\mathrm{lk}(\gamma,K)$
where $\mathrm{lk}$ denotes the linking number in $S^3$ (see \cite[3.B]{BZH}).
Hence, every \emph{abelian} representation of a knot group $\pi_K$ factors through 
$\varphi\co \pi_K\to\Z$. Here, we call $\rho$ \emph{abelian} if its image is abelian.
Therefore, we obtain for each non-zero complex number $\eta\in\C^*$ an abelian representation
$\eta^\varphi\co\pi_K\to\mathrm{GL}(1,\C) =\C^*$ given by
$\gamma\mapsto \eta^{\varphi(\gamma)}$. Notice that a $1$-dimensional representation is always irreducible.

Let $W$ be a finite dimensional $\C$-vector space.
For every representation $\rho\co G\to\GL(W)$ the vector space $W$ turns into a $\C[G]$-left module via $\rho$.
This $\C[G]$-module will be denoted by $W_\rho$ or simply $W$ if no confusion can arise. 
Notice that every finite dimensional $\C$-vector space $W$ which is a $\C[G]$-left module gives a representation 
$\rho\co G\to\GL(W)$, and by fixing a basis of $W$ we obtain a matrix representation. 
%
%

The following lemma follows from Proposition~1.7 in \cite{LM85} and the discussion therein.

\begin{lemma}\label{lem:epimorphism-representation-varieties}
Any group epimorphism $\a\co G\twoheadrightarrow F$ between finitely generated groups
induces a closed embedding
$R(F,\sl(n,\C))\hookrightarrow R(G,\sl(n,\C))$ of algebraic varieties, and an  injection
\[
X(F,\sl(n,\C))\hookrightarrow X(G,\sl(n,\C))\,.
\]
\end{lemma}

Let $H\leq G$ be a subgroup of finite index. Then the restriction of a  representation
$\rho\co G\to\SL(n,\C)$ to $H$ will be denoted by $\res_H^G\rho$ or simply by $\rho|_H$ if no confusion can arise.
This restriction is compatible with the action by conjugation and it induces a regular map
$\nu\co X(G,\SL(n,\C))\to X(H,\SL(n,\C))$.
In what follows we will make use of the following result of A.S.~Rapinchuk which follows directly from
\cite[Lemma~1]{Rap98}
\begin{lemma}\label{lem:restiction}
If $H\leq G$ is a subgroup of finite index $k$, then 
\[
\nu\co X(G,\sl(n,\C)) \to X(H,\sl(n,\C))
\]
has finite fibers. 
\end{lemma}

\subsection{The induced representation}\label{sec:induced}
Let $G$ be a group and let $H\leq G$ be a subgroup of finite index $k$. Given a representation 
$\alpha\co H\to \GL(m,\C)$ we refer to the representation of $G$ that is given by left multiplication by $G$ on
\[ \C[G]\otimes_{\C[H]} \C^m\]
as the induced representation. We denote by $e_1,\dots,e_m$ the standard basis of $\C^m$ and we pick representatives $g_1,\dots,g_k$ of $G/H$. It is straightforward to see that $g_i\otimes e_j$ with $i\in \{1,\dots,k\}$ and $j\in \{1,\dots,m\}$ form a basis for $ \C[\pi]\otimes_{\C[\ti{\pi}]} \C^m$ as a complex vector space.
Using the ordered  basis
\[ g_1\otimes e_1,\dots,g_1\otimes e_m,\dots,g_k\otimes e_1,\dots,g_k\otimes e_m\]
the induced representation can be viewed as a representation $\ind_H^G\alpha\co G\to \GL(mk,\C)$. 
If $\alpha\co H\to \SL(m,\C)$ is a representation into the special linear group, then for 
$g \in G$ a priori the determinant of $\ind_H^G\alpha(g)$ is in $\{\pm 1\}$. But it is straightforward to see that if $m$ is even,
then $\ind_H^G\alpha$ defines in fact a representation $G\to \sl(mk,\C)$.

\begin{lemma}\label{lem:induced-representations}
Let $m$ be even, and let $H\leq G$ be a subgroup of finite index $k$. Then the  map
\[ 
\iota\co R(H,\sl(m,\C))\to  R(G,\sl(mk,\C))
\]
given by $\iota(\alpha)= \ind_H^G\alpha$
is an injective algebraic map. It depends on the choice of a system of representatives, and it is compatible with the action of $\sl(m,\C)$ and $\sl(mk,\C)$ respectively.

 Moreover, the corresponding regular map $($which does not depend  
 on the choice of a system of representatives$)$
\[
\bar\iota\co X(H,\sl(m,\C))\to X(G,\sl(mk,\C))
\]
has finite fibers.
\end{lemma}

\begin{proof}
A very detailed proof of the first statement can be found in  \cite[\S10.A]{CR81} 
(see also \cite[pp.~9--10]{LM85} and \cite{Rap98}). The second part is Lemma~3 from \cite{Rap98}
\end{proof}

We are now in the position to prove our main result:

\begin{proof}[Proof of Theorem~\ref{thm:main}]
Let $K$ be a non-trivial knot. We write $\pi_K=\pi_1(E_K)$. 
Cooper, Long, and Reid~\cite[Theorem~1.3]{CLR97} (see also~\cite[Corollary~6]{Bu04}) showed that 
$G$ admits a finite-index subgroup $H$ that admits an epimorphism $\a\co H\to F_2$ onto a free group on two generators. 
It is clear, see Example~\ref{ex:free},  that $R(F_2,\sl(m,\C))\cong \SL(m,\C)^2$, and 
\[
\dim X(F_2,\SL(m,\C)) \,\,=\,\, m^2-1.
\]
It follows from Lemma~\ref{lem:epimorphism-representation-varieties} that the variety
$X(H,\SL(m,\C))$ has a component of dimension at least $m^2-1$.
We denote by $k$ the index of $H$ in $G$, and we will suppose that $m$ is even.
Then it follows  from Lemma~\ref{lem:induced-representations} that 
$X(G,\sl(mk,\C))$ contains an irreducible component of dimension at least $m^2-1$. 

Now for all $m>k$ we have $m^2-1>mk-1$. Therefore, for a given $N\in\N$ we choose
an even $m\in\N$, $m>k$, such that $n:= mk \geq N$.
The character variety $X(\pi_K,\sl(n,\C))$ contains an irreducible component whose dimension is bigger than $m^2-1>mk-1=n-1$.
\end{proof}

\section{Proof of Theorem~\ref{mainthm2}}\label{section:proof-irreducible-reps}
We let  $R^\mathit{irr}(G ,\sl(n,\C))\subset R(G ,\sl(n,\C))$ denote the Zariski-open subset of irreducible representations. The set $R^\mathit{irr}(G ,\sl(n,\C))$ is invariant by the $\SL(n,\C)$-action, and we will denote by $X^\mathit{irr}(G ,\sl(n,\C))\subset X(G ,\sl(n,\C))$ its image in the character variety. Notice that $X^\mathit{irr}(G ,\sl(n,\C))$ is an orbit space for the action of 
$\SL(n,\C)$ on $R^\mathit{irr}(G ,\sl(n,\C))$ (see \cite[Chap.~3, \S3]{Newstead1978}). 

Before we can give the proof of Theorem~\ref{mainthm2} we need to introduce several further definitions. These notations are classic (see \cite{Ser77, Bro82} for more details).

Let $H$ and $K$ be two subgroups of finite index of $G$, and let $\alpha\co H\to \mathrm{GL}(W)$ be a linear representation. Then for all $g\in G$ we obtained the \emph{twisted} representation $\alpha^g\co gHg^{-1}\to \GL(W)$ given by
 \[
 \alpha^g(x) = \alpha(g^{-1}x g), \text{ for $x\in gHg^{-1}$.}
 \]
 Notice that the twisted representation $\alpha^g$ is irreducible or semisimple if and only if $\alpha$ is irreducible or semisimple respectively.

Now, we choose a set of representatives $S$ of the $(K,H)$ double cosets of $G$. 
For $s\in S$, we let  $H_s = sHs^{-1}\cap K\leq K$. 
We obtain a homomorphism $\res^{sHs^{-1}}_{H_s}\alpha^s\co H_s\to \mathrm{GL}(W)$ by
restriction of $\alpha^s$ to $H_s$. 
The representation  $\res^G_K\ind_H^G\alpha$ is equivalent to the direct sum of twisted representations:
\begin{equation}\label{eq:res_ind}
\res^G_K \ind_H^G\alpha \,\,\cong \,\,
\bigoplus_{s\in S} \ind^K_{H_s}\res^{sHs^{-1}}_{H_s}\alpha^s\,.
\end{equation}

Equation \eqref{eq:res_ind} takes a simple form if $H=N=K$ is a normal subgroup of finite index of $G$. We obtain:
\begin{equation}\label{eq:res_ind_normal}
\res^G_N \ind_N^G\alpha \cong 
\bigoplus_{s\in S} \alpha^s
\end{equation}
where $S$ is a set of representatives of the $N$ cosets of $G$. 

In what follows we will make use of the following lemmas:
\begin{lemma}\label{lem:finite_index}
Let $G$ be a group, $H \leq G$ a subgroup of finite index, and 
$\rho\co G\to\gl(V)$ be a representation. 
If $\res^G_H\rho\co H\to \gl(V)$ is a semisimple,
then $\rho\co G\to\gl(V)$ is  semisimple.
\end{lemma}

\begin{proof}
This is Theorem 1.5 in \cite{Weh73}.
\end{proof}
\begin{lemma}\label{lem:induced_is_semisimple}
Let $G$ be a group, $H \leq G$ a subgroup of finite index, and 
$\alpha\co H\to\gl(W)$ be a representation. 
If $\alpha$ is irreducible, then $\ind^G_H\alpha$ is  semisimple.
\end{lemma}

\begin{proof}
We can choose a normal subgroup $N\unlhd G$ of finite index such that
$N\leq H$. More precisely,  we can take 
$$N=\bigcap_g gHg^{-1}$$ 
to be the \emph{normal core}  of $H$ in $G$.
We choose a set of representatives $S$ of the $(N,H)$ double cosets of $G$. 
In this case we obtain that $H_s = sHs^{-1}\cap N = s(H\cap N)s^{-1} =N$, and the double coset $NsH$ is equal to $sH$ since $N\subset H$ is normal.
Therefore, equation~\eqref{eq:res_ind} gives:
\[
\res^G_N\ind^G_H\alpha
\cong 
\bigoplus_{s\in S} \res^{sHs^{-1}}_N \alpha\,.
\]
Now, $\res^{sHs^{-1}}_N \alpha\co N \to \mathrm{GL}(W)$ is a \emph{twist}
of $\alpha|_N$ i.e. for all $g\in N$ we have
\[
 \res^{sHs^{-1}}_N \alpha(g) = \alpha(s^{-1} g s)= (\alpha|_N)^s(g)\,.
\]

By Clifford's theorem \cite[Theorem 1.7]{Weh73}, we obtain that $\alpha|_N$ is semisimple. We have that
$\alpha|_N = \alpha_1\oplus\cdots\oplus\alpha_k$ is a direct sum of simple representations.
Therefore,
\[
\big(\ind^G_H\alpha\big)\big|_N \cong 
\bigoplus_{s\in S} \alpha^s_1\oplus\cdots\oplus\alpha^s_k
\]
is the direct sum of irreducible representations. This proves that
$\big(\ind^G_H\alpha\big)\big|_N$ is semisimple, and
it follows from Lemma~\ref{lem:finite_index} that
$\ind^G_H\alpha$ is semisimple.
\end{proof}

\begin{corollary}\label{cor:induced_is_semisimple}
Let $G$ be a group,and let $N \lhd G$ a normal subgroup of finite index.
If
$\alpha\co N\to\SL(V)$ is irreducible, then $\ind^G_N (\alpha)$ is semisimple.

Moreover, if $\ind^G_N (\alpha)\cong\rho_1 \oplus\cdots\oplus\rho_l$ is a decomposition of 
$\ind^G_N (\alpha)$ into irreducible representations $\rho_j\co G\to \SL(V_j)$, then 
$\dim V$ divides $\dim V_j$ and hence
\[
\dim V\,\,\leq \,\,\dim V_j \,\,\leq \,\,\dim(V)\cdot [G:N]\,.
\]
\end{corollary}

\begin{proof}
The first part follows directly from equation \eqref{eq:res_ind_normal} and Lemma~\ref{lem:finite_index} since $\alpha^s$ is irreducible for all $s\in G$.
Notice that $S$ is now a set of representatives of the cosets $G/N$.
Moreover, we obtain
\[
\bigoplus_{s\in S} \alpha^s\,\,\cong\,\,\res_N^G \ind^G_N \alpha\,\,\cong\,\, \rho_1|_N \oplus\cdots\oplus \rho_l|_N\,.
\]
If $\rho_j|_N$ is irreducible, then it must be isomorphic to one of the twisted representations $\alpha^s$. Otherwise $\rho_j|_N$ is isomorphic to a direct sum of twisted representations $\alpha^s$, $s\in S$, and hence $\dim V_j$ is a multiple of $\dim V$.
\end{proof}

\begin{lemma} \label{lem:core}
Let $G$ be a group and let $H\leq G$ be a finite index subgroup. 

If there exists a surjective homomorphism $\varphi\co H\twoheadrightarrow F_2$ onto a free group of rank two, then there exists a normal subgroup $N\unlhd G$ of finite index 
such that $N\leq H$, and $\varphi(N)\leq F_2$ is a free group of finite rank $r\geq 2$.
\end{lemma}

\begin{proof}
Let $N$ be the normal core of $H$ i.e.\ 
$N =\bigcap_{g\in G} g H g^{-1}\unlhd G$.
The normal subgroup is a finite index subgroup of $G$, and $N\leq H$.
Now, $H\geq \varphi^{-1}(\varphi(N))\geq N$, and therefore $\varphi^{-1}(\varphi(N))$ is also of finite index in $H$ (and hence in $G$).
Hence, $\varphi(N)\unlhd F_2$ is of finite index, and $\varphi(N)$ is a free group of rank $r = [F_2:\varphi (N) ] +1 \geq 2$.
\end{proof}

The following theorem implies in particular Theorem~\ref{mainthm2} from the introduction.

\begin{theorem}
Let $K\subset S^3$ be a non-trivial knot. Then there exists $k\in\N$ such that for all 
even $m\in\N$ there exists $p\in\N$ such that $m\leq p\leq mk$, and
\[
\dim X^\mathit{irr}(\pi_K,\sl(p,\C)) \geq \frac{m^2-k}{k}\,.
\]
In particular, for $m$ even with $m>k^2$ there exists $p\in\N$ such that $m\leq p < m\sqrt{m}$, and
\[
\dim X^\mathit{irr}(\pi_K,\sl(p,\C)) \geq \frac{m^2-k}{k} > km-1 \geq p-1\,.
\]
\end{theorem}

\begin{proof}
By Lemma~\ref{lem:core} there exists a finite index normal subgroup $N\unlhd \pi_K$ of the knot group $\pi_K$, and an epimorphism $\psi\co N\twoheadrightarrow F_2 $. We put $k=[\pi_K :N]$.

For all even $m\in\N$, we obtain a regular map 
$\psi^*\co X^\mathit{irr}(F_2,\sl(m,\C))\to X(N,\sl(m,\C))$ and we let denote 
$C\subset X(N,\sl(m,\C))$ the image of $\psi^*$. By Chevalley's theorem, the set 
$C\subset X(N,\sl(m,\C))$ is  constructible (see \cite[10.2]{Borel}). Again, by Chevalley's theorem the image $D := \bar\iota(C)$,
$\bar\iota\co X(N,\sl(m,\C)) \to X(\pi_K,\sl(km,\C))$, is also a constructible set.
Notice that $\dim C = \dim D = m^2-1$ since $\phi^*$ is an embedding and $\bar \iota$ has finite fibers.

If $D$ contains a character of an irreducible representation, then 
$D\cap X^\mathit{irr}(\pi_K,\sl(km,\C))$ contains a
Zariski-open subset of $\bar D$ which is of dimension $m^2-1\geq (m^2 -k)/k$. 
Hence the conclusion of the theorem is satisfied for $p=km$.

If $D$ does not contain irreducible representations, then 
$D\subset X^\mathit{red}(\pi_K,\sl(km,\C))$.
In this case we can choose for a given $\chi\in D$ a semisimple representation $\rho$ such that $\chi_\rho=\chi$. Now, we follow the argument in the proof of Corollary~\ref{cor:induced_is_semisimple} and we obtain
 $\rho\sim \rho_1\oplus\cdots\oplus\rho_l$ where $\rho_j\co\pi_K\to\GL(p_j,\C)$ and
 $m\leq p_j < km$.
For the $l$-tuple 
$(p_1,\ldots,p_l)$
we consider the regular map
\[
\Phi_{(p_1,\ldots,p_l)}\co R(\pi_K,\sl(p_1,\C))\times\cdots\times R(\pi_K,\sl(p_l,\C))\times (\C^*)^{l-1}
\to R(\pi_K,\sl(km,\C))
\]
given by
\[ 
\Phi_{(p_1,\ldots,p_l)}(\rho_1,\ldots,\rho_l,\lambda_1,\ldots,\lambda_{l-1}) =
\bigoplus_{i=1}^{l-1} (\rho_i \otimes \lambda_i^{p_l\varphi}) \oplus 
\big(\rho_l \otimes (\lambda_1^{-p_1}\cdots\lambda_{l-1}^{-p_{l-1}})^\varphi\big)\,.
\]
The map $\Phi_{(p_1,\ldots,p_l)}$ induces a map between the character varieties
\[
\bar\Phi_{(p_1,\ldots,p_l)}\co X(\pi_K,\sl(p_1,\C))\times\cdots\times X(\pi_K,\sl(p_l,\C))\times (\C^*)^{l-1}
\to X^\mathit{red}(\pi_K,\sl(km,\C))\,.
\]

The restriction of $\bar\Phi_{(p_1,\ldots,p_l)}$ to 
\[
X^\mathit{irr}(\pi_K,\sl(p_1,\C))\times\cdots\times X^\mathit{irr}(\pi_K,\sl(p_l,\C))\times (\C^*)^{l-1}
 \]
 has finite fibers and we denote the image par $D_{(p_1,\ldots,p_l)}$. Again, by Chevalley's theorem, the image $D_{(p_1,\ldots,p_l)}\subset X^\mathit{red}(\pi_K,\sl(km,\C))$ is a constructible set, and
 \begin{equation}\label{dimd}
 \dim D_{(p_1,\ldots,p_l)} = \sum_{j=1}^l \dim X(\pi_K,\sl(p_j,\C)) + l-1\,.
 \end{equation}
 By Corollary~\ref{cor:induced_is_semisimple} $D$ is covered by finitely many sets of the form
 $D_{(p_1,\ldots,p_l)}$. Since $\dim D = m^2-1$ there must be at least one set $D_{(p_1,\ldots,p_l)}$ of dimension at least $m^2-1$. If we apply (\ref{dimd}) to this choice we obtain that 
 \[
\sum_{j=1}^l \dim X(\pi_K,\sl(p_j,\C)) \geq m^2 - l\,.
\]
In particular there exists a $j$ such that the corresponding summand is greater or equal than $\frac{m^2-l}{l}$. Note that from $m\leq p_j\leq mk$ for $j=1,\dots,l$ and $p_1+\dots+p_l=mk$ it follows that  $l\leq k$ which in turn implies that  $\frac{m^2-l}{l}\geq \frac{m^2-k}{k}$. Summarizing  we see that
\[
\dim X^\mathit{irr}(\pi_K,\sl(p_j,\C)) \,\geq\,\frac{m^2 - k}{k}\,.
\]
This concludes the proof of the first statement of the theorem.

The second statement follows from the first statement using some elementary algebraic inequalities.
\end{proof}

\section{The character variety of the figure-eight knot}\label{sec:irreducible}

The aim of this section is to prove Corollary~\ref{cor:figure-8}.
In order to study the character variety of the figure-eight knot we have to address the question under which conditions the induced representation is irreducible. This is a quite classical subject and we follow Serre's exposition in \cite{Ser77}.

\begin{lemma}[Schur's lemma and its converse]\label{lem:semisimple_vers_simple}
Let $\rho\co G\to\mathrm{GL}(W)$ be a representation.
If $\rho$ is simple, then $\mathrm{Hom}_{\C[G]}(W,W)\cong\C$. 
Conversely,  if $\rho\co G\to\GL(W)$ is semisimple
and  $\mathrm{Hom}_{\C[G]}(W,W)\cong\C$, then $\rho$ is simple.
\end{lemma}

\begin{proof}
A proof of Schur's lemma can be found in \cite[\S 27]{CR62}, and its converse is an easy exercise. Notice that the hypothesis \emph{semisimple} is essential for the converse of Schur's lemma (see \cite[p.~189]{CR62}).
\end{proof}

%
%

\subsection{The adjoint isomorphism}
%
Let $Q$, $R$ be rings and let $_QS$, $_R T_Q$, $_R U$ be modules 
(here the location of the indices indicates whether these are left or right module structures).
Then $\mathrm{Hom}_R(T,U)$ is a left $Q$-module via $qf(v) = f(v\,q)$ for all $v\in V$, and
\begin{equation}\label{eq:adjoint}
\mathrm{Hom}_Q\big(S,\mathrm{Hom}_R(T,U)\big) \cong
\mathrm{Hom}_R\big(T\otimes_Q S,U\big)
\end{equation}
(see \cite[Theorem~2.76]{Rot09}).

\begin{lemma} Let $G$ be a group and $H\leq G$ a subgroup of finite index.
 For each
$\C[H]$-left module $W$ and a $\C[G]$-left module $V$ we obtain
\begin{align}\label{eq:adjoint1}
\mathrm{Hom}_{\C[G]}\big(V,\ind_H^G  W \big) &\cong
\mathrm{Hom}_{\C[H]}\big(\res_H^G  V, W\big)\,,
\intertext{and}
\label{eq:adjoint2}
\mathrm{Hom}_{\C[H]}\big(W,\res_H^GV \big) &\cong
\mathrm{Hom}_{\C[G]}\big(\ind^G_H W , V\big)\,.
\end{align}
\end{lemma}
\begin{proof}
For proving \eqref{eq:adjoint1} we apply \eqref{eq:adjoint} with 
$Q=\C[G]$, $R=\C[H]$, $S=V$, $T=\C[G]$, and $U=W$:
\[ 
\mathrm{Hom}_{\C[G]}\big(V,\mathrm{Hom}_{\C[H]}(\C[G], W)\big) \cong
\mathrm{Hom}_{\C[H]}\big(\C[G]\otimes_{\C[G]} V, W\big)\,.
\]
Since $H\leq G$ is of finite index we obtain that the coiduced module
$\mathrm{coind}_H^G ( W) := \mathrm{Hom}_{\C[H]}(\C[G], W)$ and the induced module
$\ind_H^G ( W)$ are isomorphic as $\C[G]$-left modules (see \cite[III\,(5.9)]{Bro82}).
Moreover, $\res_H^G (U)$ and $\C[G]\otimes_{\C[G]} U$ are isomorphic as left $\C[H]$-modules,
and \eqref{eq:adjoint1} follows.

In oder to prove \eqref{eq:adjoint2} we apply \eqref{eq:adjoint} with 
$Q=\C[H]$, $R=\C[G]$, $S = W$, $T=\C[G]$ and $U = V$:
\[
\mathrm{Hom}_{\C[H]}\big(W,\mathrm{Hom}_{\C[G]}(\C[G] , V) \big) \cong
\mathrm{Hom}_{\C[G]}\big(\C[G]\otimes_{\C[H]} W , V\big)\,.
\]
The left $\C[H]$-module $\mathrm{Hom}_{\C[G]}(\C[G] , V)$ is isomorphic to $\res_H^G (V)$, and hence \eqref{eq:adjoint2} follows.
\end{proof}

\subsection{Mackey's irreducibility criterion}
Let $H\leq G$ be a subgroup of finite index, and let 
$\alpha\co H\to \mathrm{GL}(W)$ be a representation.  
For $s\in G$ we obtain the conjugate representation
$\alpha^s\co sHs^{-1}\to \mathrm{GL}(W)$. We define  $H_s := sHs^{-1}\cap H$,
and we set $W_s := \res^{sHs^{-1}}_{H_s} W_{\alpha^s}$.
In what follows we call two semisimple representations $V$ and $V'$ of $G$ \emph{disjoint} if
$\mathrm{Hom}_{\C[G]}(V,V') = 0$
\begin{lemma}[Mackey's criterion]\label{lem:irred}
Let $H\leq G$ be a subgroup of finite index.  We suppose that 
$\alpha\co H\to \mathrm{GL}(W)$ is an irreducible representation. Then
$\ind^G_H\alpha$ is irreducible if and only if for all $s\in G-H$ the $\C[H_s]$-modules
$W_s$ and $\res^H_{H_s} W$ are disjoint.
\end{lemma}
\begin{proof} It follows from Lemma~\ref{lem:induced_is_semisimple} that 
$\ind^G_H\alpha$ is semisimple. Therefore, by  
Lemma~\ref{lem:semisimple_vers_simple} it follows that  $\ind^G_H\alpha$ is irreducible if and only if  
$\mathrm{Hom}_{\C[G]}(\ind^G_H W,\ind^G_H W) \cong\C$. 
We choose a system $S$ of the $(H,H)$ double cosets of $G$. Then we obtain that 
\begin{align*}
\mathrm{Hom}_{\C[G]}(\ind^G_H W,\ind^G_H W)   & \cong
\mathrm{Hom}_{\C[H]}(W,\res^G_H \ind^G_H W), & \text{ (by \eqref{eq:adjoint2})}\\
&\cong  \bigoplus_{s\in S} 
\mathrm{Hom}_{\C[H]} (W,\ind^{H}_{H_s} W_s), & \text{ (by \eqref{eq:res_ind} for $K=H$)}\\
&\cong  \bigoplus_{s\in S} 
\mathrm{Hom}_{\C[H_s]} (\res_{H_s}^H W, W_s), & \text{ (by \eqref{eq:adjoint1}).}
\end{align*}
Now, if $s\in H$, then $H_s=H$ and $W_s \cong W$, and $\mathrm{Hom}_{H} ( W, W) \cong \C$ since $W$ is an irreducible $\C[H]$-module (see Lemma~\ref{lem:semisimple_vers_simple}).
Therefore, $\mathrm{Hom}_{\C[G]}(\ind^G_H W,\ind^G_H W)\cong \C$ if and only if 
$\mathrm{Hom}_{\C[H_s]} (\res_{H_s}^H W, W_s) =0$ for all  $s\in G-H$.
\end{proof}


\begin{corollary}
$G$ be a group,and let $N \unlhd G$ be a normal subgroup of finite index.
If
$\alpha\co N\to\SL(V)$ is irreducible, then $\ind^G_N (\alpha)$ is irreducible if and only if
$\alpha$ and $\alpha^s$ are non-equivalent for all $s\in G - N$.
\end{corollary}
\begin{proof}
We apply Lemma~\ref{lem:irred} in the case $H=N$ is a normal subgroup. It follows that $N_s=N$ and $\alpha^s$ is the twisted representation. Notice that two irreducible representations are disjoint if and only if they are not equivalent. 
\end{proof}

\subsection{Example}
Let $K=\mathfrak{b}(\alpha,\beta)\subset S^3$ be a two-bridge knot.
A presentation of the knot group $\pi_{\alpha,\beta}$ is given by 
\[
\pi_{\alpha,\beta}=\langle s,t \mid l_s s = t l_s\rangle\,\quad\text{%
where $l_s = s^{\epsilon_1}t^{\epsilon_2}\cdots t^{\epsilon_{\alpha-1}}$, and
$\epsilon_k = (-1)^{[k\mathfrak{\beta/\alpha}]}$.}
\]
We consider the following representation of $\pi_{\alpha,\beta}$ into the symmetric group:
$\delta\co \pi_{\alpha,\beta}\to\mathcal{S}_\alpha$ given by:
\begin{equation}\label{eq:dihedral_rep}
\delta(s) = (1)(2,2n+1)(3,2n)\cdots(n+1,n+2), \text{  and } \delta(a) = (1,2,\ldots,\alpha)\,,
\end{equation}
where $\alpha = 2n+1$ and $a=ts^{-1}$.
The image of $\delta$ is a dihedral group. We adopt the convention that permutations act on the right on $\{1,\ldots,\alpha\}$, and hence $\pi_{\alpha,\beta}$ acts on the right. 
We put $N=\mathrm{Ker}(\delta)$ and $H = \mathrm{Stab}(1) = \{g\in G\mid 1^{\delta(g)} = 1\}$. We have $N\leq H$, $N\unlhd \pi_{\alpha,\beta}$,
$[\pi_{\alpha,\beta}:H] =\alpha$, and $[H:N]=2$. 

The irregular  covering of $E_K$ corresponding to $H$ has been studied since the beginning of knot theory.
K.~Reidemeister calculated a presentation of $H$. Moreover,  he showed that the total space of the corresponding irregular branched covering $(\widehat{S}^3,\widehat{K})\to(S^3,\mathfrak{b}(\alpha,\beta) )$ is simply connected. He proved also that the branching set $\widehat{K}$ consists of 
$(n+1)$ unknotted components (see \cite{Rei29,Rei74}). G.~Burde proved in \cite{Bur71} that $\widehat{S}^3$ is in fact the $3$-sphere and he determined the nature on the branching set explicitly in \cite{Bur88}. 
More recently,
G.~Walsh studied the regular branched covering corresponding to $N$ \cite{Wal05}. She proved that the corresponding branching set is a great circle link in $S^3$. 
 
Let us consider the figure-eight knot $\mathfrak{b}(5,3)$ and its group $\pi_{5,3}$:
\begin{equation}\label{eq:pres_fig8_group}
\pi_{5,3}\cong\langle s,t \mid s t^{-1} s^{-1} t s = t s t^{-1} s^{-1} t\rangle
\cong \langle s,a \mid a^{-1}s^{-1}a s a^{-1} s a s^{-1} a^{-1}  \rangle 
\end{equation}
where $a=t s^{-1}$. In this case we have that $\delta\co \pi_{5,3}\to\mathcal{S}_5$ is given by
\begin{equation}\label{eq:dihedral_4_1}
\delta(s) = (1)(2,5)(3,4)\quad\text{ and }\quad \delta(a) = (1,2,3,4,5)\,.
\end{equation} 
For the figure-eight knot, the link $\widehat{K}_{4_1}$ has a particular simple form (see Figure~\ref{fig} and \cite[Example~14.22]{BZH}).
\begin{figure}[htbp]
\def\svgwidth{90pt}
\hspace*{\fill}
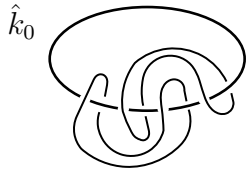 
\hspace*{\fill}
\caption{The link $\widehat{K}_{4_1}\subset \widehat{S}^3$. }
\label{fig}
\end{figure}
If we fill in the component $\hat{k}_0$, then $\widehat{K}_{4_1}$ transforms into the trivial link of two components. 
Therefore $H/\llangle y_0 \rrangle\cong F_2$ where  $\llangle y_0 \rrangle$ denotes the normal subgroup of $H$ generated by
the meridian $y_0$ of $\hat{k}_0$.

More precisely, we can use the Reidemeister--Schreier method \cite{ZVC} for finding a presentation for $H$: $\{1,a,a^2,a^3,a^4\}$ is a Schreier representative system for the
right cosets modulo $H$.
Hence generators of $H$ are 
\begin{equation}\label{eq:gen_of_H}
y_0 = s,\ y_i = a^isa^{i-5},\ i=1,2,3,4,\ \text{ and  } y_5 = a^5\,.
\end{equation}
We obtain defining relations $r_i$, $i=0,1,2,3,4$, for $H$ by expressing 
\[
r_i=a^i  (a^{-1}s^{-1}a s a^{-1} s a s^{-1} a^{-1} ) a^{-i}
\]
as a word in the $y_j$:
\[
r_0 = y_5^{-1} y_1^{-1} y_2^2 y_1^{-1},\
r_1 = y_0^{-1} y_1 y_3 y_2^{-1},\ 
r_2 = y_4^{-1} y_5 y_0 y_5^{-1} y_4 y_3^{-1},\
\]
\[
r_3 = y_3^{-1} y_4 y_0 y_4^{-1},\
r_4 = y_2^{-1} y_3 y_1 y_5 y_0^{-1} y_5^{-1}\,.
\] 
It follows that $H/\llangle y_0 \rrangle\cong\langle y_1,y_2,y_3,y_4,y_5\mid y_3, y_1=y_2, y_5\rangle
\cong F(y_1,y_4)$. Therefore, a surjection $\psi\co H\twoheadrightarrow F_2 = F(x,y)\cong $ is given by
\begin{equation}\label{eq:surj_H_to_F2}
\psi(y_0)=\psi(y_3)=\psi(y_5) = 1,\quad
\psi(y_1)=\psi(y_2)= x
\quad\text{ and }\quad
\psi(y_4) = y\,. 
\end{equation} 
We have $y_0\in H-N$, and $y_0\in\mathrm{Ker}(\psi)$.

We need also generators of $N$: we have $y_0\notin N$ hence Reidemeister-Schreier gives that $N$ is generated by
\begin{equation}\label{eq:gen_of_N}
 y_iy_0^{-1},\quad y_5, \quad y_0^2,\quad  y_0y_i,\quad y_0y_5y_0^{-1} 
 \quad\text{ where $i=1,2,3,4$.}
\end{equation}

\begin{lemma}\label{lem:alpha_is_irred}
Let $\beta\co F_2\to \SL(2m,\C)$ be irreducible. Then 
$\alpha = \beta\circ\psi\co H\to F_2\to \SL(2m,\C)$ and 
$\alpha|_N\co H\to F_2\to \SL(2m,\C)$ are also irreducible.
\end{lemma}
\begin{proof}
If $\beta\co F_2\to \mathrm{GL}(m,\C)$ is irreducible, then
$\alpha = \beta\circ\psi$ is also irreducible since $\psi\co H \to F_2$ is surjective (the representations $\alpha$ and $\beta$ have the same image). Similarly, \eqref{eq:gen_of_N} and \eqref{eq:surj_H_to_F2} give that
 $\psi|_N\co N \to F_2$ is are also surjective ($\beta$ and $\alpha|_N$ have the same image).
\end{proof}
We obtain a component of  representations $X_0\subset X(H,\sl(2n,\C))$ with
$\dim X_0 \geq 4n^2-1$, and $X_0$ contains irreducible representations.
In order to apply Lemma~\ref{lem:irred} we notice  that $\{1,a,a^2\}$ is a representative system for the $(H,H)$-double cosets 
\[
\pi_{5,3} = H \sqcup H\, a\, H \sqcup H\, a^2 \, H\,.
\]
More precisely, we have $HaH = Ha\sqcup Ha^4$, and $Ha^2H = Ha^2\sqcup Ha^3$.
We have also 
\[
H_a = H\cap a H a^{-1}= N =  H_{a^2} = H\cap a^2 H a^{-2}
\] 
since an element in
the image of the dihedral representation $\delta\co \pi_{5,3} \to \mathcal{S}_5$, given by \eqref{eq:dihedral_4_1}, 
which fixes two numbers is the identity.

For the rest of this section we let $G:=\pi_{5,3}$ denote the group of the figure-eight knot.

\begin{lemma}\label{lem:ind_alpha_is_irred}
Let $\beta\co F(x,y)\to \GL(m,\C)$ be given given by $\beta(x)=A$ and $\beta(y)=B$.
If $\beta$ is irreducible,  then 
$\rho=\ind^G_H(\beta\circ\psi)$ is irreducible.
\end{lemma}
\begin{proof}
We let $\alpha=\beta\circ\psi$ denote the corresponding representation of $H$.
By Lemma~\ref{lem:alpha_is_irred} we obtain that $\alpha|_N$, $(\alpha|_N)^a$ and
$(\alpha|_N)^{a^2}$ are irreducible. 
Hence by Lemma~\ref{lem:irred} we obtain that $\rho=\ind^G_H\alpha$ is irreducible if and only if
\[
\alpha|_N\not\sim(\alpha|_N)^a
\quad\text{ and }\quad
\alpha|_N\not\sim(\alpha|_N)^{a^2}\,.
\]
The element $y_0^2\in N$ and $\psi(y_0)=1$ implies $\alpha|_N(y_0^2) = I_m$.
Let $\beta(x) = A$ and $\beta(y) = B$ where $A,B\in\mathrm{GL}(m,\C)$.
We have
\begin{align*}
a^{-1} y_0^2 a &= a^{-1} s^2 a = a^{-5} \cdot a^4 s a^{-1} \cdot a s a^{-4}\cdot a^5 = y_5^{-1} y_4 y_1 y_5\\
\intertext{and}
a^{-2} y_0^2 a^2 &= a^{-2} s^2 a^2 = a^{-5} \cdot a^3 s a^{-2} \cdot a^2 s a^{-3}\cdot a^5 = 
y_5^{-1} y_3 y_2 y_5. 
\end{align*}
Now, $\alpha$ is given by
\[
\alpha(y_0) = \alpha(y_3)=\alpha(y_5)= I_m,\ \alpha(y_1) =\alpha(y_2) =A,\quad 
\text{ and $\alpha(y_4) = B$.}
\]
Therefore, $(\alpha|_N)^a(y_0^2) = BA$ and $(\alpha|_N)^{a^2}(y_0^2) = A$. 
Now, if $\beta$ is irreducible, then $A\neq I_n$, and $AB\neq I_n$.
Hence $\alpha|_N\not\sim(\alpha|_N)^a$ and $\alpha|_N\not\sim(\alpha|_N)^{a^2}$.
\end{proof}

\begin{proof}[Proof of Corollary~\ref{cor:figure-8}]
The subgroup $H\leq \pi_{5,3}$ is of index $5$, and $\psi\co H\to F_2$ is a surjective homomorphism onto a free group of rank $2$. 
Now, by Lemma~\ref{lem:alpha_is_irred} and Lemma~\ref{lem:ind_alpha_is_irred}, and the same argument as in the proof of Theorem~\ref{mainthm2}, we obtain that for all $m\in\N$ the character variety
$X(\pi_{5,3},\sl(10m,\C))$ has a component $C$ of dimension at least $4m^2-1$.
Finitely, Lemma~\ref{lem:ind_alpha_is_irred} implies that $C$ contains characters of irreducible representations.
\end{proof}

More explicitly, if $\beta\co F(x,y)\to \mathrm{GL}(m,\C)$ is a representation 
given by $\beta(x)=A$ and  $\beta(y)=B$ then, following the construction given in Section~\ref{sec:induced},
the induced representation
$\rho = \mathrm{ind}_H^{G} (\beta\circ\psi)\co \pi_{5,3}\to\GL(5m,\C)$ is given by
\[ 
\rho(s) =
\begin{pmatrix} 
I_m & 0 & 0 & 0 & 0 \\
0 & 0 & 0 & 0 & B \\ 
0 & 0 & 0 & I_m & 0 \\
0 & 0 & A & 0 & 0 \\
0 & A & 0 & 0 & 0
\end{pmatrix}\quad\text{ and }\quad
\rho(t) =
\begin{pmatrix} 
0 & A & 0 & 0 & 0 \\
I_m & 0 & 0 & 0 & 0 \\ 
0 & 0 & 0 & 0 & B \\
0 & 0 & 0 & I_m & 0 \\
0 & 0 & A & 0 & 0
\end{pmatrix}\,.
\]
Here, $s$ and $t$ are the generators of $\pi_{5,3}$ from \eqref{eq:pres_fig8_group}, and $I_m\in\mathrm{GL}(m,\C)$ is the identity matrix.

\subsection*{Acknowledgments.}
The authors where supported by the SFB 1085 `Higher Invariants' at the University of Regensburg, funded by the Deutsche Forschungsgemeinschaft (DFG), and by the Laboratoire de Mathématiques UMR~6620  of the University Blaise Pascal and CNRS.

\bibliographystyle{alpha}
\bibliography{ind-rep-12102016}
\end{document}